\documentclass [11pt]{amsart}
\usepackage {amsmath, amssymb, amscd, graphicx, color, url,epsf,dbnsymb}
\usepackage[all,knot,arc]{xy}
\usepackage[text={6.5in,9in},centering,letterpaper,dvips]{geometry}
\usepackage{hyphenat}

\newcommand\Zmod[1]{\mathbb{Z}/{#1}\mathbb{Z}}

\newtheorem {theorem}{Theorem}[section]
\newtheorem {lemma}[theorem]{Lemma}
\newtheorem {prop}[theorem]{Proposition}

\newtheorem {conjecture}[theorem]{Conjecture}
\theoremstyle{definition}
\newtheorem {definition}[theorem]{Definition}

\theoremstyle{remark}

\newcommand\Diagram{\mathcal D}
\newcommand\Z{\mathbb{Z}}

\def\Wedge{\Lambda}

\def\qq {{\mathbb{Q}}}

\def\rk {{\operatorname{rank}}}
\def\fin\qedhere
\def\pr {{\text{pr}}}

\def\cm{\cdot}

\def\Edges{\mathcal{E}}
\def\Crossings{\mathcal X}

\newcommand\rWedge{\Wedge_\circ}
\newcommand\Kh{\mathrm{Kh}'}
\newcommand\oKh{\mathrm{Kh}}
\newcommand\roKh{\overline \oKh}
\newcommand\rKh{\overline \Kh}
\newcommand\prKh{\overline Kh'}
\newcommand\rC{\overline C}
\newcommand\rrC{{\overline C}^{(p)}}

\newcommand\orient{\mathfrak o}

\newcommand{\Cat}{\mathfrak C}
\newcommand{\Obj}{\mathrm{Ob}}

\newcommand{\Tor}{\mathrm{Tor}}

{\makeatletter\gdef\reallynopagebreak{\nopagebreak\@nobreaktrue}}

\begin{document}

\title{Odd Khovanov homology}

\author[Peter S. Ozsv\'ath]{Peter Ozsv\'ath}
\thanks {PSO was supported by NSF grant number DMS-0505811 and FRG-0244663}
\address {Department of Mathematics, Columbia University\\ New York, NY 10027}
\email {petero@math.columbia.edu}

\author[Jacob Rasmussen]{Jacob Rasmussen}
\thanks{JR was supported by NSF grant number DMS-0603940 and a Sloan
  Fellowship}
\address{Department of Mathematics, Princeton University\\ Princeton,
  New Jersey 08544 \newline \phantom{xx} and DPMMS, University of Cambridge, UK}
\email {jrasmus@math.princeton.edu}

\author[Zolt{\'a}n Szab{\'o}]{Zolt{\'a}n Szab{\'o}}
\thanks{ZSz was supported by NSF grant number DMS-0704053 and FRG-0244663}
\address{Department of Mathematics, Princeton University\\ Princeton, New Jersey 08544}
\email {szabo@math.princeton.edu}

\begin {abstract} 
  We describe an invariant of links in \(S^3\) which is closely
  related to Khovanov's Jones polynomial homology. Our construction
  replaces the symmetric algebra appearing in Khovanov's definition
  with an exterior algebra. The two invariants have the same reduction
  modulo $2$, but differ over \(\mathbb{Q}\). There is a reduced
  version which is a link invariant whose graded Euler characteristic
  is the normalized Jones polynomial.
\end {abstract}

\maketitle
\section{Introduction}
\label{sec:Introduction}
\maketitle

In his influential paper, Khovanov~\cite{Khovanov} describes a link
invariant which associates to a link a bigraded Abelian group whose
graded Euler characteristic is the Jones polynomial.  His invariant is
obtained from a TQFT which associates to a collection of embedded,
planar circles the symmetric algebra of the vector space generated by
the circles.  Our goal here is to describe a modified version of
Khovanov homology, which associates to a collection of embedded,
planar circles the exterior algebra of the vector space generated by
the circles. 

We explain this invariant. The constructions described here are
clearly quite closely related to Khovanov's.  Indeed, the mod two
reductions of the two theories coincide (cf. Proposition~\ref{prop:ModTwoReduction} below).

\subsection{A projective TQFT}

Consider the category ${\mathfrak C}$ of compact one-manifolds and compact,
orientable cobordisms between them. The starting point for the
construction of the odd Khovanov homology is a ``projective  functor'' \(F\)
from  \({\mathfrak C}\) to the category of graded \(\Z\)--modules. This
functor is ``projective'' in the sense that the map assigned to a 
morphism is well-defined only  up to an overall sign.

An object $S\in\Obj(\Cat)$ is a disjoint union of circles. 
Given such an \(S\), we let $V(S)$ denote the free abelian group
 generated by its components, and define \(F(S) = \Wedge^*
 V(S)\). 
Morphisms in \(\Cat\) are generated by four types of
elementary morphism: zero-handle additions (or {\it births}),
one-handle additions where the feet of the one handle lie on two
different components ({\it merges}), one-handle additions where the
feet lie on the same component ({\it splits}), and two-handle
additions ({\it deaths}). We will be most interested in the morphisms
corresponding to one-handle additions. 

Suppose that $M: S_1 \to S_2 $ is a  merge cobordism which
joins the circles $a_1$ and $a_2$ in \(S_1\). Then there is a natural
identification $V(S_2)\cong V(S_1)/(a_1-a_2)$, and we define
\begin{equation*}
F_M: \Wedge^*V(S_1) \to \Wedge ^*V(S_2)
\end{equation*}
to be the map induced by the projection \(V(S_1) \to V(S_1)/(a_1-a_2)
\cong V(S_2)\). 

Next, suppose that $M: S_1 \to S_2 $ is a  split cobordism in
which a single circle \(a\) in \(S^1\) divides to form two circles
\(a_1\) and \(a_2\) in \(S_2\). Then the  natural
identification $V(S_1)\cong V(S_2)/(a_1-a_2)$
induces an identification
$$\Wedge^* V(S_1) \cong \Wedge^*
\left(\frac{V(S_2)}{(a_1-a_2)}\right)\cong (a_1-a_2)\wedge \Wedge^*
V(S_2).$$ The second isomorphism is well-defined only up to sign.
We define 
\(F_M: \Wedge^*V(S_1) \to \Wedge ^*V(S_2) \)
to be the composition
$$\begin{CD}
\Wedge^*(V(S_1))@>{\cong}>> \Wedge^*\left(\frac{V(S_2)}{(a_1-a_2)}\right)
@>{\cong}>> (a_1-a_2)\wedge \Wedge^* V(S_2)
@>{\subset}>> \Wedge^*V(S_2).
\end{CD}
$$
This map sends $1\in \Wedge^* V(S_1)$ to $\pm(a_1-a_2) \in  \Wedge^*
V(S_2)$. 

For completeness, we record the maps induced by the birth and death
cobordisms as well. A birth cobordism \(M:S_1 \to S_2\) induces an
inclusion \(V(S_1) \to V(S_2)\), and \(F_M\) is the induced map. In a
death cobordism \(M:S_1 \to S_2\), there is a distinguished component
\(A\) in \(S_2\) which is capped off by the two-handle; the map
\(F_M\) is given by contraction with the dual of \(a\). 

The reader can easily verify that up to sign, these maps satisfy all
the identities associated with a TQFT. We could eliminate the sign
ambiguity by tensoring with \(\Z/2\); the resulting TQFT can be
identified with the \(\Z/2\) reduction of the TQFT used by Khovanov in
\cite{Khovanov}. 

\begin{figure}
\begin{center}
\mbox{\vbox{\epsfbox{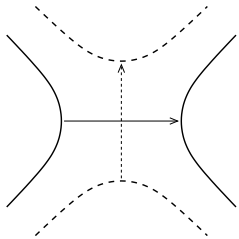}}}
\end{center}
\caption {{\bf Elementary saddle move.}
  This illustrates an elementary saddle move. The initial picture is
  indicated by solid lines: the two solid arcs represent one
  resolution, and the dotted arcs represent the other resolution.  The
  solid oriented arc represents the initial arc, and the dotted one is
  to be rotated $90^\circ$ counterclockwise.}
\label{fig:Move}
\end{figure}

Another approach is to try to deal with the sign
ambiguity by decorating our cobordisms. 
Suppose we have a cobordism $M:S_1 \to S_2$ with a fixed decomposition into
one-handles. Such a cobordism is specified by an $n$-tuple of embedded
zero-spheres $Z_1,...,Z_n\subset S_1$. These can be represented
diagrammatically by joining the two points of \(Z_i\) by an arc
representing the core of the one-handle and
fixing an (arbitrary) orientation on it, as shown in
Figure~\ref{fig:Move}. If the cobordism associated to the handle
addition \(Z_i\) is a split, we fix the sign of \(F_{Z_i}\) by
requiring that it take \(1\) to \(a_1-a_2\), where the arrow points
from \(a_1\) to \(a_2\). 

Given two one-handles $\{Z_1,Z_2\}$, the induced maps commute up to sign:
\begin{equation}
  \label{eq:DefAC}
  F_{Z_2}\circ
  F_{Z_1}=\epsilon\cm F_{Z_1}\circ F_{Z_2}.
\end{equation}
When the composite map is nontrivial, the sign of \(\epsilon\)
can be determined from the combinatorics of how the two arcs $Z_1$ and
$Z_2$ interact with each other, as illustrated in
Figure~\ref{fig:Squares}.  When $\epsilon=-1$, we call the pair {\em
  of Type $A$}; when $\epsilon=+1$, we call the pair {\em of Type
  $C$}. There are two remaining cases where the double-composites are
trivial.  We label these cases by $X$ and $Y$, as shown in
Figure~\ref{fig:Squares}.

\begin{figure}
\begin{center}
\mbox{\vbox{\epsfbox{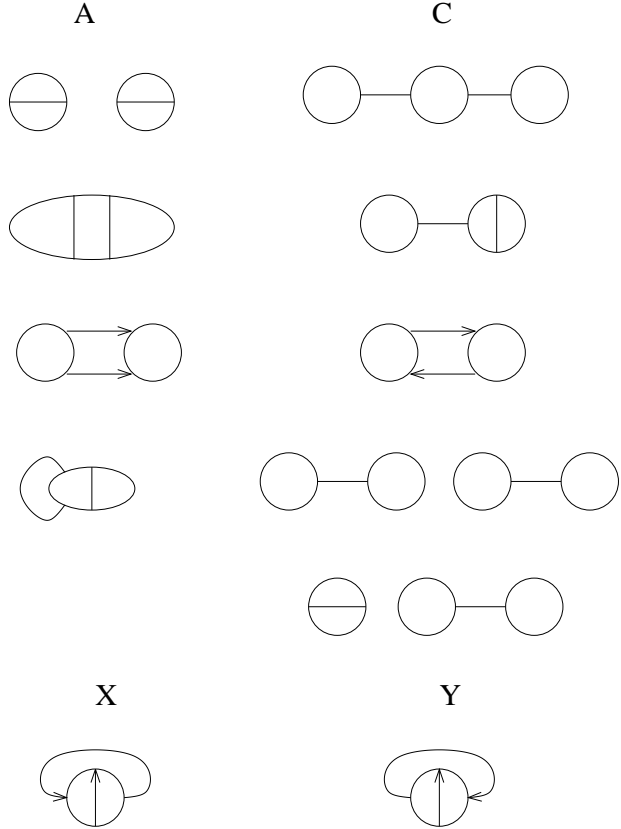}}}
\end{center}
\caption {{\bf Commutation chart.}
  The set of pairs of arcs can be placed into four categories, Types
  $A$, $C$, $X$, and $Y$. In Type $A$, the
  double-composite maps anti-commute. In Type $C$, they commute. The
  two  remaining cases are labelled $X$ and $Y$ as above.
  The thicker curves denote components of $S_1$, while the thinner arcs
  represent the one-handles specified by $Z$.  Orientations of these
  one-handles are specified by arrows when they are needed; when they
  are dropped, it is because the corresponding picture has the stated
  type for either choice of orientation.}
\label{fig:Squares}
\end{figure}

\begin{figure}
\begin{center}
\mbox{\vbox{\epsfbox{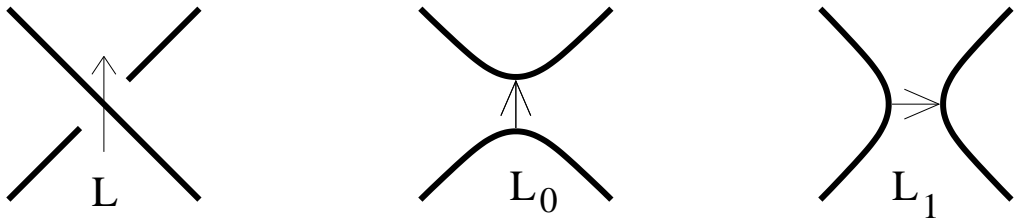}}}
\end{center}
\caption {{\bf Oriented skein moves.}
  The oriented 
crossing on the left has two oriented resolutions, as illustrated.
Note that there are two possible choices of
orientation at each crossing. (To see the other one, turn the figure
\(180^\circ\).)}
\label{fig:OrientedSkein}
\end{figure}

\subsection{The hypercube of resolutions}

Given a link projection, we fix an {\it orientation} on
it by
drawing an arrow at each crossing, as illustrated in
Figure~\ref{fig:OrientedSkein}. 
If \(\Diagram\) is  an oriented link diagram of
this form, we can define its associated
{\em hypercube of oriented resolutions}. Specifically, each crossing
in the projection has two resolutions  $L_0$ and
$L_1$. (These are the resolution conventions of~\cite{Khovanov}; 
they are opposite to those from ~\cite{BrDCov}.)
 The crossing can be thought of as giving a cobordism from $L_0$
to $L_1$ consisting of a single one-handle, as illustrated in
Figure~\ref{fig:OrientedSkein}. 
 Let $\Crossings$ denote the set
of crossings of the diagram $\Diagram$. Then for each map $I\colon
\Crossings \longrightarrow \{0,1\}$, we have an associated embedded
one-manifold in the plane, $\Diagram(I)$ which replaces a given
crossing $x\in \Crossings$ by its oriented resolution of type $I(x)$.

Given two vertices $I_{0}, I_{1}\colon \Crossings \longrightarrow
\{0,1\}$ in the hypercube, we say that there is an oriented edge from
$I_{0}$ to $I_{1}$ if there is some $x\in \Crossings$ with the
property that
\[
\begin{array}{ll}
I_{0}(x)=0        & I_{1}(x)=1 \\
I_{0}(y)=I_{1}(y) & {\text{if $x\neq y$}.} \\
\end{array}
\]
We let $\Edges(\Diagram)$ and \(\mathcal{V}(D)\) denote the set of
edges and vertices, respectively. Each 
vertex \(I\) 
 corresponds to an object \(\Diagram(I)\) in $\Cat$, and each edge in
$\Edges$ corresponds to a morphism.  Given $e\in\Edges$ from
$I_{0}$ to $I_{1}$, we let $M_{e}$ denote the corresponding morphism
from $\Diagram(I_{0})$ to $\Diagram(I_{1})$.

Two-dimensional faces, or squares, in the hypercube of oriented resolutions
correspond to pairs 
of resolutions
$I_{00},I_{11}\colon \Crossings \longrightarrow \{0,1\}$ for which there are exactly two 
$x_1, x_2\in \Crossings$ with 
\begin{eqnarray*}
I_{00}(x_1)=0,&  I_{11}(x_1)=1  \\
I_{00}(x_2)=0,& I_{11}(x_2)=1  \\
I_{00}(y)=I_{11}(y) & {\text{if $y\not\in\{x,x'\}$}} 
\end{eqnarray*}
Each square face may be classified as one of types  $A$, $C$, $X$, or $Y$,
 according to the
classification scheme in Figure~\ref{fig:Squares}.

\begin{definition}
        An {\em  edge assignment for a diagram $\Diagram$} is a
        map $\epsilon\colon \Edges(\Diagram)\longrightarrow \{\pm
        1\}$. Given an edge assignment \(\epsilon\), we say that a
        square face is {\em even} or {\em odd}, depending on whether it contains
        an even or odd number of edges with \(\epsilon(e)=-1\).
        A {\em type $X$ edge assignment} is an edge
        assignment with the property that all faces of type \(A\) and
        \(X\) are even and all faces of type \(C\) and \(Y\) are odd.
        Similarly, a {\em type $Y$ edge assignment} is an edge
        assignment for which faces of type \(A\) and
        \(Y\) are even and faces of type \(C\) and \(X\) are odd.

\end{definition}

The following lemma will be established in
Section~\ref{sec:EdgeAssignments}.

\begin{lemma}
        \label{lemma:ExistsAssignment}
        Any diagram $\Diagram$ has an edge assignment of type $X$, and
        one of type $Y$.  
\end{lemma}

\subsection{The chain complex}
Consider the vector space
$$C(\Diagram)=\bigoplus_{I\in \mathcal{V(\Diagram)}} 
\Wedge^* V(\Diagram(I)).$$
Given a type $X$ or a type $Y$ edge assignment $\epsilon\colon
\Edges(\Diagram)\longrightarrow \{\pm 1\}$, we can define an
endomorphism \(\partial_\epsilon\) of $C(\Diagram)$ by the formula
$$\partial_\epsilon(v)
=\sum_{\{e\in\Edges(\Diagram), J\in \mathcal{V}(D)
\big| \text{$e$ goes from $I$ to $J$}\}} \epsilon(e)\cm F_e(v).$$
for \(v \in \Wedge^*V(\Diagram(I))\).

We claim that  $\partial_\epsilon^2=0$. Indeed, this relation 
is satisfied provided
  $$\epsilon(e_1)\epsilon(e_2) F_{e_1}\circ F_{e_2}
  +\epsilon(e_3)\epsilon(e_4) F_{e_3}\circ F_{e_4}=0$$
  whenever we
  have four edges $\{e_1,...,e_4\}$ bounding a square (so that
  $e_1e_2$ and $e_3e_4$ are the two paths
  from the initial point to the final point).  For a square of type
  $A$ or $C$, this holds by Equation~\eqref{eq:DefAC}, while for
  squares of type $X$ or $Y$, the relation holds since
  $$F_{e_1}\circ F_{e_2}=F_{e_3}\circ F_{e_4}=0.$$
Thus the pair \((C(\Diagram),\partial_{\epsilon})\) defines a chain complex.

  $C(\Diagram)$ can be 
equipped with a bigrading $C_{a,b}(\Diagram)$,
following~\cite{Khovanov}.  Specifically, we endow the exterior
algebra $\Wedge^* V(\Diagram(I))$ with the $q$-grading $Q_0$ for which
$\Wedge^r V(\Diagram(I))$ has $q$-grading equal to $\dim V(\Diagram(I))-2r$.
Similarly, we define the initial homological grading $M_0$ on
$C(\Diagram)$ so that $C(\Diagram(I))$ is supported in grading
$\sum_{c\in\Crossings} I(c)$.  The $q$-grading on $C(\Diagram)$ is
then given by $Q=Q_0+n_+-2n_-+M_0$, where here $n_-$ denotes the number of
negative crossings in the diagram, and the homological grading
$M=M_0-n_-$.  We write
$$C(\Diagram)=\bigoplus_{m,s\in\Z} C_{m,s}(\Diagram),$$
where $m$
corresponds to the homological grading and $s$ the $q$-grading. Since
the differential preserves $Q$-grading and drops $M$-grading by one,
the two gradings descend to homology, and we can write
$$H(\Diagram)=\bigoplus_{m,s\in\Z} H_{m,s}(\Diagram).$$

\begin{theorem}
  \label{thm:KnotInvariant}
  Let $L$ be a link. Fix an oriented projection $\Diagram$ of $L$ and an edge
  assignment $\epsilon$ of type $X$ or $Y$.  The bigraded homology
  groups of $(C(\Diagram),\partial_\epsilon)$ are independent of the
  choice of \(\Diagram\) and \(\epsilon\).
\end{theorem}

We call the above bigraded homology groups the {\em odd Khovanov
  homology} of the link $L$, $\Kh(L)$, to distinguish it from ordinary
$sl(2)$ Khovanov homology $\oKh(L)$ (where the variables are
``even'').
 We collect here some properties  of \(\Kh(L)\) 
which follow quickly from its construction; proofs will
be supplied in Section~\ref{sec:Properties}.

Recall that the unnormalized
Jones polynomial is characterized by the properties that:
\begin{eqnarray*}
{\widehat J}(\emptyset)&=& 1 \\
{\widehat J}\Big({\text{(unknot)}}\cup L\Big) &
= & (q+q^{-1})\cm {\widehat J}(L) \\
{\widehat J}(\backoverslash) &\dot{=}& {\widehat J}(\hsmoothing) - q \cm {\widehat J}(\smoothing),
\end{eqnarray*}
 where for $f,g\in \Z[q,q^{-1}]$, we write $f\dot{=}g$ if
$f=\pm q^j\cdot g$ for some $j\in\Z$.

\begin{prop}
  \label{prop:EulerCharacteristic}
  $\Kh(L)$ categorifies the unnormalized Jones polynomial,
   in the sense that
  \begin{eqnarray*}
    {\widehat J}(L) = \sum_{m,s\in\Z} (-1)^m \rk (\Kh_{m,s}(L)) \cm q^s 
  \end{eqnarray*}
\end{prop}

Exactly as in Khovanov's original construction, the skein relation
characterizing the Jones polynomial is replaced by a {\em skein exact
  sequence}, cf.~\cite{Khovanov}:

\begin{prop}
  \label{prop:SkeinExactSequence}
  There is a long sequence
  $$\begin{CD}
    @>>>  \Kh(\hsmoothing) @>{i_*}>> \Kh(\backoverslash) @>{\pi_*}>>
    \Kh(\smoothing) @>{\partial}>> \Kh(\hsmoothing) @>>>
    \end{CD}.$$
\end{prop}

The maps \(i_*\), \(\pi_*\), and \(\delta\) are all homogenous with
respect to the bigrading; for a precise statement of the grading
shifts, see {\it e.g.} \cite{Thin}, \cite{RasmussenSurvey}. 

\begin{prop}
  \label{prop:ModTwoReduction}
  The mod two reduction of $\Kh(L)$ agrees with the mod two reduction
  of Khovanov's $sl(2)$ homology $\Kh(L)$; i.e. 
  $$\Tor(\oKh_{m+1,s}(L),\Zmod{2})\oplus
  \left(\oKh_{m,s}(L)\otimes\Zmod{2}\right) \cong
  \Tor(\Kh_{m+1,s}(L),\Zmod{2})\oplus
  \left(\Kh_{m,s}(L)\otimes\Zmod{2}\right).$$
\end{prop}

Despite their  formal similarities, \(\Kh\) and \(\oKh\) are actually
very different groups. The first indication of this fact is given by 
\begin{prop}
  \label{prop:Reduction}
  There is a bigraded Abelian group $\rKh(L)$ with the property that
  $$\Kh_{m,s}(L)\cong \rKh_{m,s-1}\oplus \rKh_{m,s+1}(L).$$
\end{prop}
We call $\rKh(L)$ the {\em reduced odd Khovanov homology}.  The
bigraded group $\rKh(L)$ categorifies the ordinary Jones polynomial
\(J(L)\), which is defined by the relation $(q+q^{-1})\cm
J(L)={\widehat J}(L).$ It analogous to the reduced Khovanov homology
$\roKh$ defined in \cite{ReducedHomology} which also categorifies the
ordinary Jones polynomial, but there are some differences.  In the
definition of the ordinary reduced Khovanov homology, one fixes a
component of $L$; different choices of component can lead to different
answers (which can be seen, for example, by considering the disjoint
union of the trefoil and an unknot).  By contrast, $\rKh$ is a link
invariant.  Moreover, the relation between $\rKh$ and $\Kh$ is
simpler than the relation between $\roKh$ and $\oKh$.
                     
For small knots \(K\), the groups \(\rKh(K)\) and \(\roKh(K)\) are
isomorphic. Indeed, we have
\begin{prop}
If \(L\) is a non-split alternating link, then \(\rKh(L) \cong \roKh(L)\).
\end{prop}
\noindent However, there are many nonalternating knots for which the two groups
are not isomorphic. 

In~\cite{BrDCov}, it is shown that there is a spectral sequence whose
$E_2$ term is Khovanov homology of a link $L\subset S^3$, with
coefficients taken in $\Zmod{2}$, and which converges to the Heegaard
Floer homology ${\widehat{\mathrm{HF}}}$ of the branched double-cover of $L$. Our
motivation for finding odd Khovanov homology came from our attempts to
lift this to a result over $\Z$; and consequently, it is natural to make
the following:

\begin{conjecture}
  Let $L\subset S^3$ be a link.  There is a spectral sequence whose
  $E_2$ term is the reduced odd Khovanov homology of $L$ and whose
  $E_\infty$ term is the Heegaard Floer homology of the branched
  double-cover of $L$ (with coefficients in $\Z$).
\end{conjecture}

A similar result should hold, with a suitable construction from
Seiberg-Witten monopole Floer homology~\cite{KMbook} replacing the
Heegaard Floer homology of the branched double-cover.

This paper is organized as follows. In
Section~\ref{sec:EdgeAssignments} we construct the edge assignments
needed to construct the chain complex
(Lemmas~\ref{lemma:ExistsAssignment}), and show that the isomorphism
class of the complex is independent of the choice of sign assignment
and the orientations.  In Section~\ref{sec:Invariance}, we show that
odd Khovanov homology is independent of the link projection. The
arguments here follow closely the invariance proof of Khovanov's
$sl(2)$ theory (see ~\cite{Khovanov}, see also~\cite{BarNatan},
\cite{KhovanovRozansky}, ~\cite{KhovanovRozanskyII}). In
Section~\ref{sec:Properties}, we establish the basic properties of
this construction enumerated above.  Finally, in
Section~\ref{sec:Calculations}, we exhibit some calculations of these
groups.

We wish to thank Mikhail Khovanov, Tomasz Mrowka, and Paul Seidel for
their encouragement during the preparation of this manuscript.

\section{Existence and uniqueness of edge assignments}
\label{sec:EdgeAssignments}

Our first goal in this section is to prove 
Lemma~\ref{lemma:ExistsAssignment},
which allows us to construct the chain complex for odd Khovanov
homology. We then make some preliminary steps towards the proof of
Theorem~\ref{thm:KnotInvariant} by showing that the isomorphism type
of the complex \((C(\Diagram),\partial_{\epsilon})\) does not depend 
 on \(\epsilon\) or on the choice of orientation at the crossings.

Lemma~\ref{lemma:ExistsAssignment} will follow quickly from the
following lemma about cubes in the hypercube of oriented resolutions;
but before stating this lemma, we note that a cube in the hypercube of
resolutions is determined by a pair of resolutions
$$I_{000},I_{111}\colon \Crossings \longrightarrow \{0,1\}$$ with the
property that there are three crossings $x_1,x_2,x_3\in\Crossings$ 
such that
\[\begin{array}{lll}
I_{000}(x_i)=0, & I_{111}(x_i)=1 & {\text{for $i=1,2,3$}} \\
I_{000}(y)=I_{111}(y)& & {\text{if $y\not\in\{x_1,x_2,x_3\}$}}.
\end{array}
\]

\begin{lemma}
  \label{lemma:CountingInCube}
  Each cube in the hypercube of resolutions contains an even number of
  squares of type $A$ and $X$.  Similarly, each cube contains an even
  number of squares of type $A$ and $Y$.
\end{lemma}

\begin{proof}
  This is a case-by-case analysis according to the different possible
  combinatorial types of cubes in the hypercube of resolutions.
  Specifically, a cube corresponds to eight resolutions, which are
  given a partial ordering, with a unique minimal element.  The cube
  is determined by this minimal element ($I_{000}$) and the three
  oriented arcs connecting various components (corresponding to three
  crossings in the original projection). We disregard all the
  unknotted circles which do not meet these three arcs. This leaves us
  between one and four circles, which are connected by the oriented
  arcs (in the plane). We enumerate the possible connected diagrams in
  Figure~\ref{fig:Cubes}.

\begin{figure}
\begin{center}
\mbox{\vbox{\epsfbox{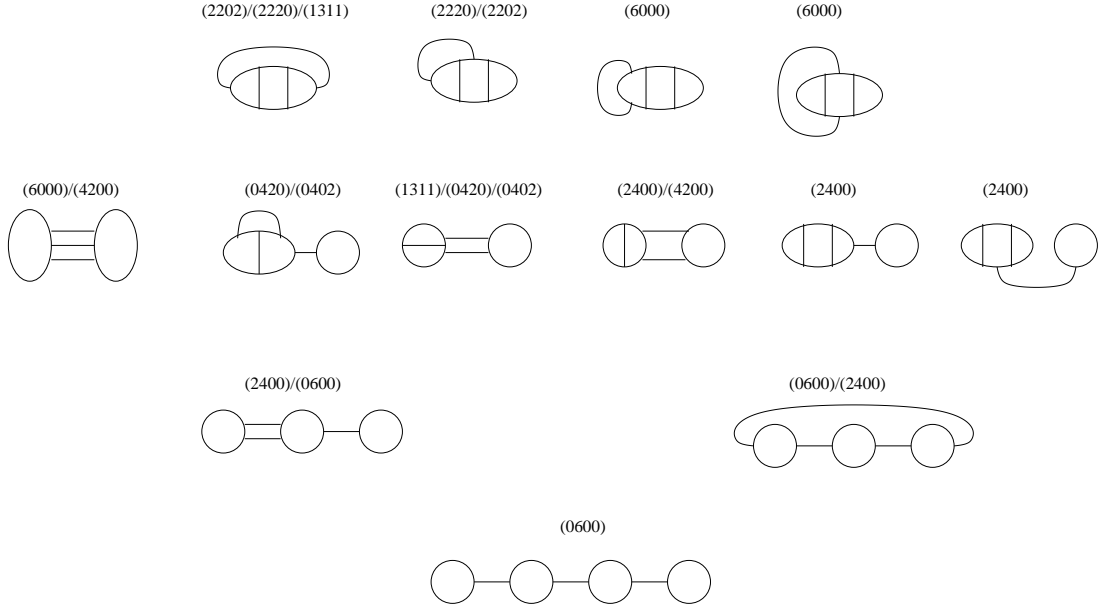}}}
\end{center}
\caption {{\bf Cubes.}
        Different types of cubes in the hypercube of resolutions.
        Each cube is labelled with choices
        of distributions of faces of type $A$, $C$, $X$, and $Y$.
        The different possibilities correspond to various orientations
        on the arcs connecting them.}
\label{fig:Cubes}
\end{figure}

        Let $a$, $c$, $x$, and $y$ denote the number of squares of
        types $A$, $C$, $X$ and $Y$ respectively in each cube. We
        claim that  both $a+x$ and $a+y$ are even.  This,
        too, is an easy verification. Note that the six squares in each
        cube are realized by choosing one of the three arcs, and
        either dropping it, or performing surgery along it.  For
        example, in Figure~\ref{fig:Example}, we have pictured the
        possibilities for the second type of cube appearing in
        Figure~\ref{fig:Cubes}, with one of the eight different possible
        choices of orientations. We see that there are two squares of
        type $C$, two of type \(A\), and two of type \(Y\). 
        The number of squares of types \(C\), \(A\), \(X\), and \(Y\)
        in the other cases is indicated in figure
        ~\ref{fig:Cubes}. We leave it to the reader to verify that
        in all cases, \(a+x\) and \(a+y\) are even. 

  If the diagram is disconnected, it has a component with only one
  arc. We orient the cube so that the four edges corresponding
  to this arc are vertical; then the top and bottom faces are of the same
  type. If the vertical edges correspond to  merges, all four vertical
  faces are of type \(C\). If they are splits, each merge in the top
  face corresponds to a vertical face of type \(C\), and each split
  corresponds to a vertical face of type \(A\). The number of merges
  and splits in the top face are both even, so the claim holds in this
  case as well.

\begin{figure}
\begin{center}
\mbox{\vbox{\epsfbox{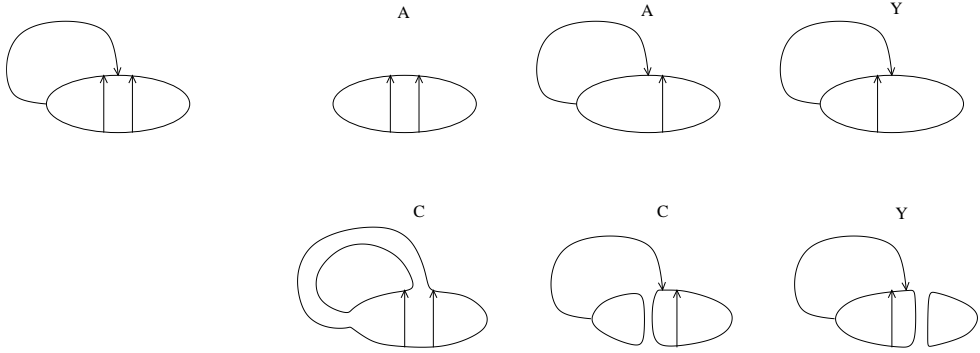}}}
\end{center}
\caption {{\bf Counting types of squares in a cube.}
Consider the cube corresponding to the diagram on the left-hand side
of the figure. 
The six square faces of the cube are gotten by
either forgetting or  performing surgery along one of the arcs. 
These are shown and classified on the right.}
\label{fig:Example}
\end{figure}

\end{proof}

\begin{proof}[Proof of Lemma~\ref{lemma:ExistsAssignment}]
 Let \(G=\Z^* \cong \Z/2\) be the multiplicative group with elements
 \(\pm1\). 
  We consider the hypercube $Q$ of oriented resolutions as a
  simplicial complex. We can define on it a  $2$-cochain  $\phi\in
  C^2(Q;G)$  which
  associates to each face of type $A$ or $X$ the element
  $1\in G$, and to each face of type $C$ or $Y$ the
  number $-1\in G$.
  Lemma~\ref{lemma:CountingInCube} shows that $\phi$ is a cocycle.
  Since the cube is contractible, $\phi$ must be a coboundary.
  Concretely, this means that there is some function $\epsilon\colon
  \Edges\longrightarrow \{\pm 1\}$ with the property that
  $\epsilon(e_1)\epsilon(e_2)\epsilon(e_3)\epsilon(e_4)=\phi(\sigma)$
  where $e_1,...,e_4$ are the four edges of the square $\sigma$. This
  is the required edge assignment of type $X$. The same remarks
  hold for constructing an edge assignment of type $Y$.
\end{proof}

\begin{lemma}
  \label{lemma:Uniqueness}
 If $\epsilon$ and $\epsilon'$ are two edge assignments of the same
  type (\(X\) or \(Y\)), then the chain complex
  $(C(\Diagram),\partial_\epsilon)$ is isomorphic to
  $(C(\Diagram),\partial_{\epsilon'})$.
\end{lemma}

\begin{proof}
  Suppose  that $\epsilon$ and $\epsilon'$ are two edge assignments
  of the same type. Then $\epsilon\cm \epsilon'$ is a one-dimensional
  cocycle, so it can be
  realized as the co-boundary of a zero-cochain; i.e.  we have a map
  $\eta\colon \mathcal{V}(\Diagram) 
  \longrightarrow \{\pm 1\}$ with
  $\eta(v_1)\eta(v_2)=\epsilon(e)\epsilon'(e)$ if $v_1$ and $v_2$ are
  the endpoints of $e$.  Consider the endomorphism $\Phi\colon
  C(\Diagram)\longrightarrow C(\Diagram)$
   which, when restricted to $C(\Diagram(I))$,
  is given by multiplication by $\eta(I)$. It is straightforward to
  verify that $\Phi$ is an isomorphism of chain complexes, from
  $(C(\Diagram),\partial_\epsilon)$ to
  $(C(\Diagram),\partial_{\epsilon'})$.

\end{proof}

\begin{lemma}
  \label{lemma:IndepOfOrientation}
  If \(\Diagram\) and \(\Diagram'\) are two oriented diagrams with the
  same underlying diagram but different orientations, then there are
  edge assignments of the same type \(\epsilon\) and \(\epsilon'\)
  with \((C(\Diagram),\partial_{\epsilon}) \cong (C(\Diagram'
  ),\partial_{\epsilon'})\).
\end{lemma}

\begin{proof}
It suffices to consider the case where we change the orientation at a
single crossing \(c\).
If we denote the maps in the new cube of resolutions by
\(F_{e}'\), then  we can write \(F_e' = \alpha(e) \cdot F_{e}\). 
Here \(\alpha(e)
= -1\) if the crossing associated to \(e\) is \(c\) and the
corresponding cobordism is a split, and \(\alpha(e) = 1\) otherwise.
If \(\phi'\) is the class in \(C^2(Q;G)\)
associated to the new cube, we claim that 
\( \phi' = \phi\cdot d\alpha. \) For faces of type \(A\) and \(C\),
this is obvious, while for a face \(\sigma\) of type \(X\) or
\(Y\), 
reversing the orientation of one of the two arcs in \(\sigma\) switches
types \(X\) and \(Y\), so \(\phi'(\sigma) = -\phi (\sigma)\). On the
other hand, exactly one of the two edges associated to \(c\) is a
split (and thus has \(\alpha(e) = -1\).) It follows that if
\(\epsilon\) is an edge assignment of type \(X\) for the old cube,
\(\alpha \cdot \epsilon\) is an edge assignment of type \(X\) for the
new cube, and the boundary map \(\partial_{\alpha \cdot \epsilon}\) is
in the new complex is exactly the same as \(\partial_{\epsilon}\) in
the old one. 
\end{proof}

\begin{lemma}
  If $\epsilon$ and $\epsilon'$ are sign assignments of opposite
  types, then there is an isomorphism $(C,\partial_\epsilon)\cong
  (C,\partial_{\epsilon'})$.
\end{lemma}

\begin{proof}
  We divide the crossings of $L$ into two equivalence classes as
  follows.  Fix one of the two checkerboard colorings of the diagram.
  We can then define a function $\theta$ from the crossings into $\pm 1$
  depending on how the crossing is colored, as illustrated in 
  Figure~\ref{fig:CheckerBoard}. 

  \begin{figure}
    \begin{center}
      \mbox{\vbox{\epsfbox{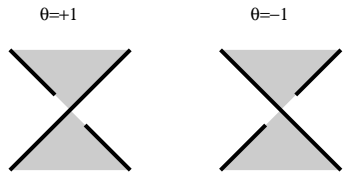}}}
    \end{center}
    \caption {{\bf Black and white coloring.}
      Given a checkerboard coloring of the projection, we can introduce a
      function $\theta\colon \Crossings\longrightarrow \{\pm 1\}$ as
      illustrated.}
    \label{fig:CheckerBoard}
  \end{figure}
  
  Fix some initial orientation $\orient$ for the crossings in a
  knot projection, and a sign assignment $\epsilon$. Consider next a
  different set of orientations on the initial crossings, specified by
  $\theta\cm \orient$ (i.e. if $c$ is some crossing with type
  $\theta(c)=+1$, then ${\orient'}(c)$ is the same as
  ${\orient}(c)$, whereas if $\theta(c)=-1$, then the two
  orientations point in opposite directions).  Clearly, this change of
  orientations swaps squares of type $X$ (for one orientation) with
  those of type $Y$ (for the other).  Moreover, it preserves the
  types of all other squares.  (Note that there are only two squares
  not of type $X$ or $Y$ whose types depend on the orientations of the
  two arcs; for those two squares, the arcs represent crossings in the
  same equivalence class.) Thus, if we view $\epsilon$ as a type $X$
  sign assignment for orientation $\orient$, then $\epsilon$ can
  also be viewed as a type $Y$ assignment for the orientation
  $\orient'$. The lemma now follows from
  Lemma~\ref{lemma:IndepOfOrientation}.
\end{proof}

\section{Topological invariance}
\label{sec:Invariance}

In this section, we check that $\Kh(L)$ is invariant under the three
Reidemeister moves, thus verifying Theorem~\ref{thm:KnotInvariant}. 
The argument is more or less the same one used by Khovanov 
to prove invariance of the ordinary $sl(2)$ homology~\cite{Khovanov}, see
also~\cite{KhovanovRozansky}, ~\cite{KhovanovRozanskyII}.  
We follow Bar-Natan's
exposition in~\cite{BarNatan}.

\begin{prop}
  \label{prop:RI}
  The homology groups of $C(\Diagram)$ remain invariant as the 
  diagram undergoes a Reidemeister move of Type I.
\end{prop}

\begin{proof}
  Suppose that $\Diagram'$ is obtained from $\Diagram$ by a
  Reidemeister move of type I, as shown in
  Figure~\ref{fig:ReidemeisterI}. For concreteness, we focus on the
  Reidemeister move illustrated in the top row.

  \begin{figure}
    \begin{center}
      \mbox{\vbox{\epsfbox{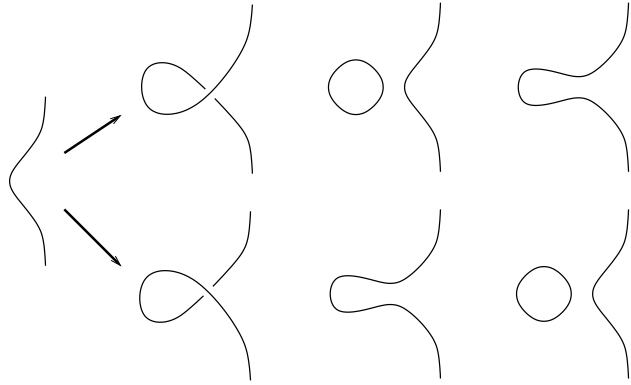}}}
    \end{center}
    \caption {{\bf Reidemeister I.}
      Starting with the initial picture on the left, we can perform
      two types of Reidemeister I move, as illustrated in the next column.
      After that, we display the $0$- and $1$-resolutions respectively.}
    \label{fig:ReidemeisterI}
  \end{figure}

  We can split the $\Z$-module $C(\Diagram')\cong C(\Diagram_0)\oplus
  C(\Diagram_1)$, where $\Diagram_0$ denotes  the
  disjoint union of $\Diagram$ with an unknotted component $O$, and
  $\Diagram_1$ is identified with the original diagram $\Diagram$.
  Indeed, if $\epsilon$ is an edge assignment for $\Diagram$, we can
  use it to induce an edge assignment $\epsilon'$ on $\Diagram'$ by
  declaring its restriction to $\Diagram_1$ to be identified with
  $-\epsilon$; its restriction to $\Diagram_0$ to agree with
  $\epsilon$, and its assignment to the edges connecting $\Diagram_0$
  to $\Diagram_1$ to be all $+1$. The map $\epsilon'$ is an edge
  assignment of the same type as $\epsilon$, since all the squares
  involving the edges connecting $\Diagram_0$ and $\Diagram_1$ are of
  type $C$ (there is exactly one edge leaving the distinguished
  unknotted component).

  In effect, we have identified $C(\Diagram')$ with the mapping cone of a map
  $$
  \begin{CD}
    D\colon C(\Diagram_0) @>>> C(\Diagram_1),
  \end{CD}$$
  where 
  \begin{eqnarray*}
    (C(\Diagram_0),\partial_0)\cong (C(\Diagram),\partial)\otimes
    \Wedge^* (\Z\cm v_0)
    &{\text{and}}&
    (C(\Diagram_1),\partial_1)\cong (C(\Diagram),-\partial),
  \end{eqnarray*}
  and the map $D$  identifies $v_0$ with $v_1$
  (where $v_0$ denotes the unknotted component in $\Diagram_0$ and $v_1$
  denotes the component which is connected to it). 
  This chain complex is clearly quasi-isomorphic to the subcomplex 
  $(v_1-v_0)\wedge (C(\Diagram_0),\partial_0)$, which in turn is
  identified with the complex $C(\Diagram)$.
  This establishes the stated isomorphism for one of the two types
  of Reidemeister move.
\end{proof}

\begin{prop}
  \label{prop:RII}
  The homology groups of $C(\Diagram)$ remain invariant as the 
  diagram undergoes a Reidemeister move of Type II.
\end{prop}

\begin{proof}
  We consider the diagram for the second Reidemeister move, following
  notation suggested in Figure~\ref{fig:ReidemeisterII}.  The
  resolutions after the Reidemeister move have four types,
  $\Diagram_{i,j}$ with $i,j\in\{0,1\}$, so that $\Diagram_{00}$ and
  $\Diagram_{11}$ are identified, $\Diagram_{10}$ is obtained from
  $\Diagram_{00}$ by inserting an unknotted component, and
  $\Diagram_{01}$ is identified with the diagram before the
  Reidemeister move. The chain complex after the Reidemeister move can
  be written in the following form:
$$\begin{CD}
  C(\Diagram_{0,1})@>{d^{*,1}}>>C(\Diagram_{1,1}) \\
  @A{d^{0,*}}AA @AA{v_2\sim v_3}A \\
  C(\Diagram_{0,0})@>{\wedge (v_1-v_2)}>> C(\Diagram_{1,0})
 \end{CD}$$ 
 (we have labelled $v_1$ to be the component belonging to the top
 of the Reidemeister II move, $v_2$ to be the middle component, and
 $v_3$ to be the bottom). In defining this complex, we have
 implicitly picked a sign assignment \(\epsilon'\) for \(\Diagram'\)
 (the diagram after the second Reidemeister move.) It is easy to see that
 the restriction of \(\epsilon'\) to \(C(\Diagram_{1,0})\) will be a
 sign assignment for \(\Diagram\). 

  \begin{figure}
    \begin{center}
      \mbox{\vbox{\epsfbox{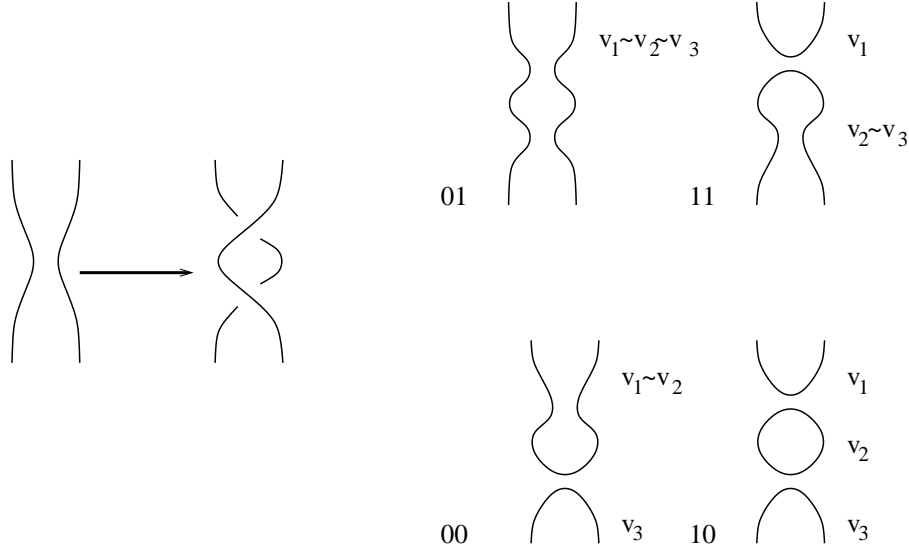}}}
    \end{center}
    \caption {{\bf Reidemeister II.}
      Starting with the initial picture on the left, we can perform a
      Reidemeister move to obtain the second picture. We have then
      illustrated (and labelled) the four resolutions of this diagram.}
    \label{fig:ReidemeisterII}
  \end{figure}

Let \(X \subset C(\Diagram_{0,1})\) be the kernel of the contraction
map with \(v_2^*\). (In other words, \(X\) is generated by those
exterior products which do not contain \(v_2\) as a factor.) It
is not difficult to see that \(A = X \oplus C(\Diagram_{1,1})\) is an
acyclic subcomplex of \(C(\Diagram')\). The quotient complex
\(C(\Diagram')/A\) will have the same homology as
\(C(\Diagram')\). This quotient is of the form
$$\begin{CD}
  C(\Diagram_{0,1}) \\
  @A{d^{0,*}}AA  \\
  C(\Diagram_{0,0})@>{\wedge (v_1-v_2)}>> C(\Diagram_{1,0})/X
 \end{CD}$$ 
This complex clearly has \(C(\Diagram_{0,1})\) as a
subcomplex, and it is easy to check that the quotient 
\((C(\Diagram')/A)/C(\Diagram_{0,1})\) 
is acyclic. Thus \(C(\Diagram_{0,1})\) will have
the same homology as \(C(\Diagram')\). 

\end{proof}
\begin{prop}
  \label{prop:RIII}
  The homology groups of $C(\Diagram)$ remain invariant as the diagram
  undergoes a Reidemeister move of Type III. 
\end{prop}
\begin{proof}
  To see how, we consider the cube of resolutions before the
  Reidemeister move of type III is performed. This is illustrated in
  Figure~\ref{fig:RIII}.  After contracting the indicated arrows, we
  obtain an intermediate chain complex illustrated in
  Figure~\ref{fig:RIIIB}.  This in turn is related to the cube of
  resolutions for the projection after the Reidemeister move is
  performed, by contracting two edges as indicated in
  Figure~\ref{fig:RIIID}.

  \begin{figure}
    \begin{center}
      \mbox{\vbox{\epsfbox{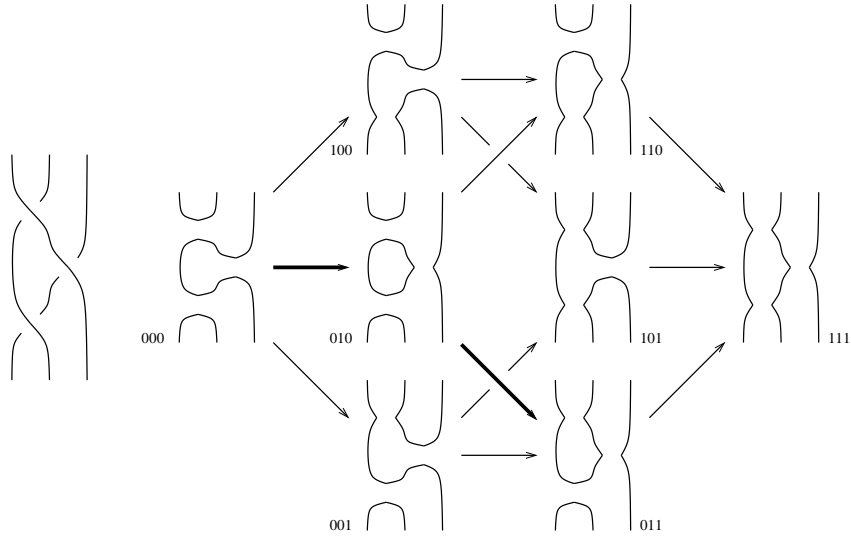}}}
    \end{center}
    \caption {{\bf Reidemeister III.}
      Consider the local picture on the left. We obtain a corresponding
      cube of resolutions which is pictured on the right, where the resolutions
      are labelled by their corresponding resolution vector.  The thick arrow
      is contracted to obtain the chain complex from
      Figure~\ref{fig:RIIIB} below.}
    \label{fig:RIII}
  \end{figure}

  \begin{figure}
    \begin{center}
      \mbox{\vbox{\epsfbox{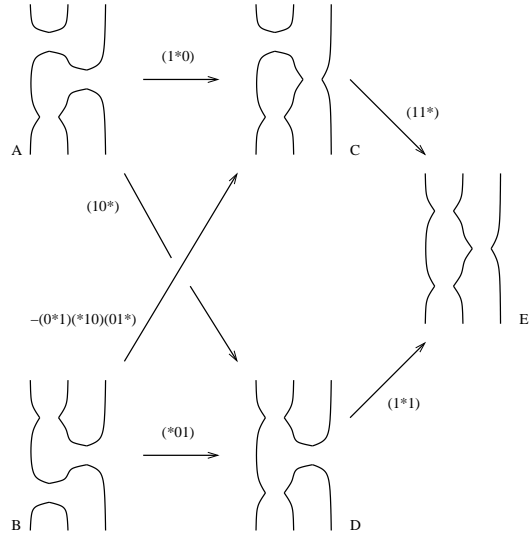}}}
    \end{center}
    \caption {{\bf Reidemeister III, after contraction.}
      }
    \label{fig:RIIIB}
  \end{figure}

  We give a more precise version of this argument, with signs.
  Consider the diagram illustrated in Figure~\ref{fig:RIII}.
  This complex has a quotient, consisting of
  a mapping cone of a map 
  $$d\colon C(000)\longrightarrow C(010)\wedge \mu,$$
  where here $\mu$ denotes the unknotted component in the resolution
  indicated by the vector $010$.
  Clearly, this map $d$ is an isomorphism, and hence the full complex
  from Figure~\ref{fig:RIII} is isomorphic to the subcomplex
  illustrated on the left in Figure~\ref{fig:RIIIC}.

  \begin{figure}
    \begin{center}
      \mbox{\vbox{\epsfbox{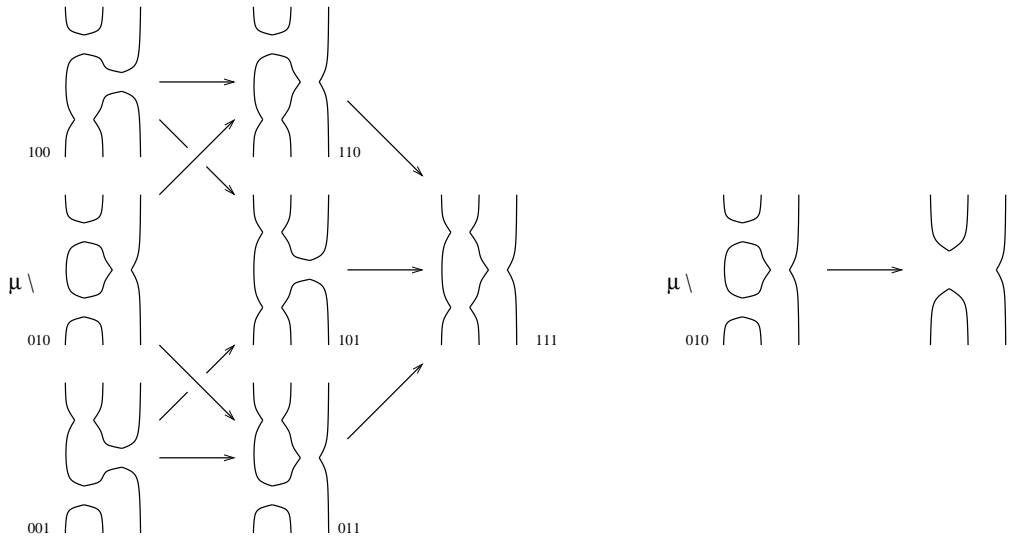}}}
    \end{center}
    \caption {{\bf Reidemeister III, intermediate states.}
      On the left is the quotient complex after the rightmost thick arrow
      from Figure~\ref{fig:RIII} is contracted. There is a natural map to 
      the mapping cone illustrated on the right.}
    \label{fig:RIIIC}
  \end{figure}

  Observe that \(C(110) \cong C(011)\). Fix a sign assignment
  \(\epsilon_{110}=\epsilon_{011}\) on these two cubes, and extend it to a sign
  assignment \(\epsilon\) on the entire cube of resolutions. (This is
  possible, since the quotient space in which we collapse \(C(110)\)
  and \(C(011)\) to a point has vanishing \(H^2\).)
  The five term complex $P$ shown in Figure~\ref{fig:RIIIB} is
  obtained by multiplying the edge maps by the indicated signs. For
  example, the differential from $A$ to $D$ is multiplied by the sign
  $\epsilon(10*)$ from the hypercube of Figure~\ref{fig:RIII};
  similarly, the differential from $B$ to $C$ is multiplied by the
  sign $-\epsilon(0\!*\!1)\epsilon(01*)\epsilon(*10)$.
  
  Consider the map $\Phi$ from the complex on the left in
  Figure~\ref{fig:RIIIC} to the complex in Figure~\ref{fig:RIIIB}
  defined as follows.  $\Phi$ induces the natural identification from
  $C(100)$ to $A$; $\Phi$ induces the natural identification from
  $C(001)$ to $B$; $\Phi$ induces the identification of $C(110)$ to $C$;
  $\Phi$ induces the identification of $C(011)$ with $C$ times
  $-\epsilon(*10)\epsilon(01*)$; $\Phi$ induces induces the
  identification of $C(101)$ with $D$; $\Phi$ induces induces the
  identification of $C(111)$ with $E$. Finally, the restriction of $\Phi$
  to $C(010)$ is trivial. It is straightforward to verify that $\Phi$
  is a chain map.
   
  \begin{figure}
    \begin{center}
      \mbox{\vbox{\epsfbox{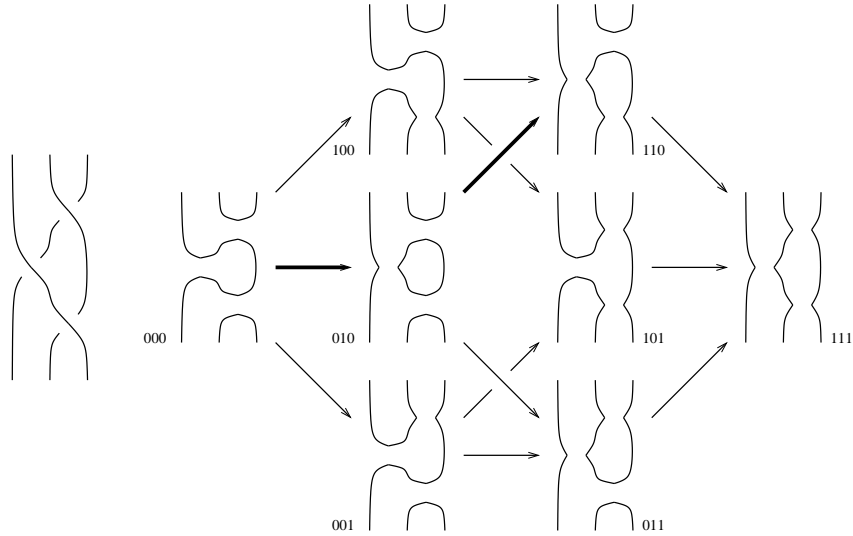}}}
    \end{center}
    \caption {{\bf After Reidemeister III.}
      This is the complex after Reidemeister III is performed on
      Figure~\ref{fig:RIII}. Contracting the thick arrows, we obtain a
      five term complex which is isomorphic to the one from Figure~\ref{fig:RIIIB}.}
    \label{fig:RIIID}
  \end{figure}
  
  We now define a chain map $\Psi$ from the right complex $R$ to the left
  complex $L$.  Note that the right complex consists of two terms, one of
  which is identified with $C(010)/\mu$, and the second of which is
  naturally identified with $C(110)$ or $C(011)$. The differential
  is an isomorphism (and hence the right complex is acyclic).

  Restricted to $C(010)/\mu$, the map \(\Psi\) is
  an isomorphism onto the corresponding term in the cube of
  resolution; its restriction to the second term is identified with
  the diagonal identification onto $C(110)$ and $C(011)$; in fact, we
  multiply the component in $C(110)$ with the sign of the edge
  $\epsilon(1\!*\!0)$ and the component in $C(011)$ with the sign of the
  edge $\epsilon(0\!*\!1)$.  It is straightforward to verify that $\Psi$
  is a chain map. Moreover, it is straightforward to verify the short exact sequence
  $$
  \begin{CD}
    0@>>>R @>{\Psi}>> L @>{\Phi}>> P@>>> 0.
  \end{CD}
  $$
  This completes the proof that the homology of the complex before
  the Reidemeister move is identified with the complex pictured in
  Figure~\ref{fig:RIIIB}. 
  
  We now give a more intrinsic description of the
   signs in the chain
  complex  of Figure~\ref{fig:RIIIB}.  To this end, consider the
  simplicial complex ${\mathcal P}$ obtained as the product of the
  two-dimensional simplicial complex appearing in
  Figure~\ref{fig:RIIIB} (5 vertices, 6 edges, and 2 faces)
 with a cube of the appropriate dimension.
  It is easy to see that  \({\mathcal
    P}\) is contractible. Let $\psi$ be the
  two-dimensional cochain on ${\mathcal P}$ with $\Zmod{2}$
  coefficients which takes
  non-trivial values on faces of type $C$ or $Y$ (as in the proof of
  Lemma~\ref{lemma:ExistsAssignment}.) Above, we have described an
  explicit sign assignment \(\epsilon\) on the edges of \({\mathcal
    P}\) for which \(d \epsilon = \psi\).  
  As in Lemma~\ref{lemma:Uniqueness},
  $H^1({\mathcal P},\Zmod{2})=0$, so  if $\epsilon_1$ and $\epsilon_2$ are
  any two one-dimensional cochains with
  $\delta\epsilon_1=\delta\epsilon_2=\psi$, then they determine
  isomorphic chain complexes. 

Now consider  the diagram we obtain after making
 a Reidemeister III move. This diagram and the associated chain
complex are shown in Figure~\ref{fig:RIIID}. We can contract the
boldface edges (just as we did with the complex in
Figure~\ref{fig:RIII}) to obtain a new chain complex which agrees (up
to sign) with \(P\). To show that the two complexes are genuinely
isomorphic, it is enough to check that the corresponding obstruction
cocycles \(\psi,\psi' \in C^2({\mathcal P}, \Z/2)\) are the same. If
we choose the upward pointing orientation at all of the crossings in
Figures~\ref{fig:RIII} and~\ref{fig:RIIID} (and use the same
orientations at crossings not shown in the diagram), then 
this is easily seen to be the case.
\end{proof}

\begin{proof}[Proof of Theorem~\ref{thm:KnotInvariant}]
  According to a classical result of Reidemeister, an invariant of
  link projections which is unchanged under the three Reidemeister
  moves is in fact a link invariant, cf. for
  example~\cite{BurdeZieschang}.
(Note that it is actually sufficient to check invariance under the
three moves considered here, since  the ``other'' Reidemeister I and III
moves can be obtained by composing the moves we have studied  with some
Reidemeister II moves.)
  Thus, $\Kh$ is a link invariant according to
  Propositions~\ref{prop:RI},~\ref{prop:RII}, and~\ref{prop:RIII}.
\end{proof}

\section{Basic properties}
\label{sec:Properties}

In this section, we sketch the proofs  of the properties 
 stated in the introduction. We begin by giving two
equivalent definitions of the reduced homology. 

Let $\rWedge^* V(\Diagram)$ be the subalgebra of 
$\Wedge^* V(\Diagram)$ generated by the kernel of the map
$V(\Diagram)\longrightarrow \Z$ defined by 
$$\sum n_i a_i \mapsto \sum n_i.$$
There is a corresponding subcomplex $\rC(\Diagram)\subset
C(\Diagram)$.

Given a generic point $p$ on the knot projection, we also
have a subalgebra
$$a_p \wedge \Wedge^* V(\Diagram(I))\subset \Wedge^* V(\Diagram(I)),$$
where $a_p$ denotes the component of \(\Diagram(I)\) containing $p$.
There is a corresponding subcomplex 
$$\rrC(\Diagram)=a_p\wedge C(\Diagram)\subset C(\Diagram).$$

\begin{lemma}
  Wedge product with $a_p$ induces an isomorphism
  of $\rWedge^* V(\Diagram)$ with $a_p \wedge \Wedge^* V(\Diagram)$.
  In fact, this induces an isomorphism of complexes
  $\rrC(\Diagram)\cong \rC(\Diagram)$.
\end{lemma}

\begin{proof}
  Straightforward.
\end{proof}

\begin{proof}[Proof of Proposition~\ref{prop:Reduction}]
  If \(\Diagram \) is a diagram of the link \(L\), we define
 $\rKh(L)$ to be $H_*(\rC(\Diagram))$.
   Consider the disjoint union of $\Diagram$ with an unknotted link $O$.
  Taking $p$ to lie on the unknotted component, it is easy to see that
  $\rrC(\Diagram\cup O)\cong C(\Diagram)$. On the other hand, taking
  $p$ elsewhere, we see that $\rrC(\Diagram\cup O)\cong
  \rrC(\Diagram)\oplus \rrC(\Diagram)$.  Since $\rrC(\Diagram)\cong
  \rC(\Diagram)$, and the latter chain complex is independent of the
  placement of the basepoint, we conclude that
  $C(\Diagram)\cong \rC(\Diagram)\oplus\rC(\Diagram).$
  Passing to homology, we obtain the 
  stated result.
\end{proof}

\begin{proof}[Proof of Proposition~\ref{prop:ModTwoReduction}.]
  The mod two reduction of $\Wedge^* V(\Diagram)$, equipped with our
  multiplication and comultiplication maps (whose mod two reduction
  no longer depends on the choice of orientations at the crossings) coincides
  with the mod two reduction of Khovanov's TQFT.  It follows at once that if
  ${\mathrm{CKh}}(\Diagram)$ denotes Khovanov's complex, then
  ${\mathrm{CKh}}(\Diagram)\otimes \Zmod{2}\cong
  C(\Diagram)\otimes\Zmod{2}$. 
 The
  proposition now follows from the universal coefficient theorem.
\end{proof}

\begin{proof}[Proof of Proposition~\ref{prop:EulerCharacteristic}]
  This a direct consequence
  of Proposition~\ref{prop:ModTwoReduction} and the analogous
   formula for the ordinary Khovanov homology.
\end{proof}

\begin{proof}[Proof of Proposition~\ref{prop:SkeinExactSequence}]
  By construction, $\Kh(L)$ is the homology of a mapping cone of
  $C(L_0)$ to $C(L_1)$. The stated exact sequence is an application of
  the long exact sequence of a mapping cone, with appropriate shifts
  in gradings. A corresponding result for $\rKh$ is also apparent,
  using the definition of the reduced complex $\rC$.
\end{proof}

\section{Calculations}
\label{sec:Calculations}

In this section, we give a few computations of the odd Khovanov
homology.
In light of Proposition~\ref{prop:Reduction}, we may as well restrict our
attention to the reduced  groups \(\rKh(L)\).
For the simplest knots and links, these groups exhibit a pattern which
is familiar from the usual Khovanov homology:

\begin{definition}
 \(\rKh(L)\) is said to be {\em \(\sigma\)--thin} if \(\rKh(L)_{m,s} = 0\)
whenever \(s-2m \neq \sigma(L)\). 
\end{definition} 

Here, our sign convention for the signature is that positive links have
positive signature. In analogy with Lee's theorem \cite{EunSooLee} on the
ordinary cohomology of alternating knots, we have 

\begin{prop}
If \(L\) is an alternating link, then  \(\rKh(L)\) is \(\sigma\)-thin.
\end{prop}

\begin{proof}
The standard proof of this result for the ordinary Khovanov homology
relies on two facts: first,  that the reduced homology 
\(\roKh\) satisfies a skein exact
sequence like the one in Proposition~\ref{prop:SkeinExactSequence},
and second, that \(\roKh\) of the unknot 
is supported in bigrading \((0,0)\). Since
both of these hold for \(\rKh\) as well, the proof goes through
without change.  
\end{proof}
\noindent More generally, the same result holds if \(L\) is 
{\it quasi-alternating} in the sense of \cite{BrDCov} ({\em c.f.}
\cite{Thin}). 

In light of Proposition~\ref{prop:EulerCharacteristic}, 
it is not difficult to see that if \(\rKh(L)\) is \(\sigma\)-thin, it
is completely determined by the Jones polynomial and signature of
\(L\). Since the same result is true for \(\roKh\), we see that
\(\rKh(L) \cong \roKh(L)\) whenever \(L\) is alternating. 
The analogous statement for the unreduced homology is
emphatically not true; \(\roKh\) and \(\oKh\) are related by a long
exact sequence analogous to Proposition~\ref{prop:Reduction}
$$\begin{CD}
@>>> \roKh_{s,m+1}(L) @>>> \oKh_{s,m}(L)@>>>\roKh_{s,m-1}@>\partial>>
\roKh_{s+1,m+1}(L) @>>>
\end{CD}$$
but the boundary map in this sequence is almost never \(0\). The
difference between the two is already evident with the trefoil knot,
for which \(\Kh\) has rank 6 and \(\oKh\) has rank 4. 

To find examples where \(\rKh(L) \neq \roKh(L)\), we  resort to
computer calculations. Using a {\it Mathematica} program based on
Bar-Natan's original program for computing the Khovanov homology
\cite{BarNatan}, we computed  \(\rKh(K)\otimes \qq\)
 for all nonalternating knots \(K\) with fewer than \(12\) crossings.
The first knot which is not quasi-alternating (and thus the first for
which we might  expect the two to differ) is the \((3,4)\) torus knot,
number \(8_{19}\) in the Rolfsen table. Indeed, we find that
\(\rKh(8_{19})\) has rank \(3\), with graded Poincar{\'e} polynomial 
\begin{align*}
\prKh(8_{19})(q,t) = \sum_{m,s} t^m q^s \dim \rKh(8_{19}) \otimes \qq
= q^6 + q^{10}t^2 + q^{16}t^5,
\end{align*}
while \(\roKh(8_{19})\) is known to have rank \(5\) \cite{KhoHo}, 
with graded Poincar{\'e} polynomial 
\begin{align*}
\overline{Kh}(8_{19})(q,t)  = \sum_{m,s} t^m q^s \dim \roKh(8_{19})
\otimes \qq 
 = q^6 + q^{10}t^2 +q^{12}t^3 + q^{12}t^4+ q^{16}t^5.
\end{align*}
This computation is somewhat disappointing, since it indicates that
\(\Kh\) cannot possess a cancelling differential analogous to the Lee
differential \cite{EunSooLee} and its generalizations \cite{Turner}. 

The other non-quasi-alternating knot with fewer than 10 crossings has
Rolfsen number \(9_{42}\). \(\rKh(9_{42})\otimes \qq\) 
turns out to be \(\sigma\)-thin. However, the corresponding statement
over \(\Z\) cannot be true,  since if we use  \(\Z/2\) coefficients,
 \(\rKh(9_{42})\)  reduces to
 \(\roKh(9_{42})\), which is not \(\sigma\)-thin. These two examples
 exhibit a general trend which  continues with the 10 and 11-crossing
 knots; 
namely, that the rank of \(\rKh(K)\) tends to be
 smaller than that of \(\roKh(K)\). The table below compares the
 dimensions of \(\rKh \otimes \qq \) and \(\roKh \otimes \qq \) for the
 non-alternating 10-crossing knots, omitting those knots for which
 both groups are \(\sigma\)--thin. 

$$
\begin{array}{|c|c|c||c|c|c||c|c|c|}
\hline
\rule{0pt}{14pt} K &  \thinspace  \rKh(K) & 
\thinspace  \roKh(K)  &
 K &  \thinspace  \rKh(K) &  \thinspace  \roKh(K)
 & K &  \thinspace  \rKh(K) & \thinspace  \roKh(K)  \\

\hline
\hline
10_{124} & 3 & 7 &
10_{139} & 7 & 11 &
10_{153} & 9 & 17 
\\
10_{128} & 11 & 13 &
10_{145} & 7 & 13 &
10_{154} & 17 & 21 
\\
10_{132} & 5 & 11 &
10_{152} & 15 & 19 &
10_{161} & 9 & 13 
\\
10_{136} & 15 & 17 &
& & &
& & \\
\hline
\end{array}
$$
\vskip0.07in

Among the 10-crossing knots,
  \(\rKh\) is \(\sigma\)--thin whenever 
 which \(\roKh\) is \(\sigma\)--thin; 
when \(\roKh\) is not \(\sigma\)--thin, the
 rank of \(\rKh\) is strictly smaller. The same pattern also
 holds for the 11-crossing knots.
Below, we tabulate the
 Poincar{\'e} polynomials of those 10-crossing knots for which 
\(\rKh\)  is not \(\sigma\)-thin.

\begin{align*}
\overline{Kh'}(10_{124}) & = q^8 + q^{12} t^2 + q^{20} t^7  \\
\overline{Kh'}(10_{139}) & = q^8 + q^{12} t^2 + q^{16} t^5 +
q^{18} t^6 + q^{20} t^7 + q^{22} t^8 + q^{24} t^9 \\
\overline{Kh'}(10_{145}) & =  q^{-20} t^{-9} + q^{-18}t^{-8} + 
q^{-16} t^{-7} +  q^{-14} t^{-6} + q^{-8} t^{-3} + q^{-8} t^{-2} +
q^{-4} \\
\overline{Kh'}(10_{152}) & =  q^{-26} t^{-10} + 2 q^{-24} t^{-9} + 2
q^{-22} t^{-8} + 3 q^{-20} t^{-7}  \\ &  \phantom{XXX} + 2q^{-18} t^{-6} +   2q^{-16} t^{-5}
+ q^{-14} t^{-4}  + q^{-12} t^{-2} + q^{-8} \\
\overline{Kh'}(10_{153}) & = 
q^{-10} t^{-5} + q^{-8} t^{-4} + q^{-6} t^{-3} + q^{-4} t^{-2} + 1 + q^2 t^2 +
  q^4 t^3 + q^6 t^4 + q^8 t^5 \\
\overline{Kh'}(10_{154}) & = q^6 + q^{10} t^2 + q^{10} t^3 + 2 q^{12} t^4 + 2 q^{14} t^5 + 2 q^{16} t^6 + 
 3 q^{18} t^7 + 2 q^{20} t^8 + 2 q^{22} t^9 + q^{24} t^{10} \\
\overline{Kh'}(10_{161}) & =
 q^{-22} t^{-9} +q^{-20} t^{-8} + q^{-18} t^{-7}   + 
 q^{-16} t^{-6}  \\ & \phantom{XXX} + q^{-14} t^{-5} + q^{-12} t^{-4} + q^{-10} t^{-3} +
 q^{-10} t^{-2} + q^{-6}
\end{align*}

\bibliographystyle{plain}
\bibliography{biblio}

\end{document}